\documentclass{article}
\usepackage{amsfonts,amsmath,amsthm,amssymb}
\usepackage{graphics, epsfig}
\usepackage{color}
\usepackage{appendix}
\usepackage{ulem}
\usepackage[makeroom]{cancel}
\newenvironment{ack}{\medskip{\it Acknowledgement.}}{}

 \usepackage[usenames,dvipsnames]{pstricks}
 \usepackage{pst-grad} 
 \usepackage{pst-plot} 
\allowdisplaybreaks

\let\TeXchi\chi
\newbox\chibox
\setbox0 \hbox{\mathsurround0pt $\TeXchi$}
\setbox\chibox \hbox{\raise\dp0 \box 0 }
\def\chi{\copy\chibox}

\newtheorem{proposition}{Proposition}[section]
\newtheorem{theorem}{Theorem}[section]
\newtheorem{lemma}{Lemma}[section]
\newtheorem{corollary}{Corollary}[section]
\newtheorem{remark}{Remark}[section]

\numberwithin{equation}{section}
\numberwithin{theorem}{section}
\numberwithin{proposition}{section}
\numberwithin{lemma}{section}
\numberwithin{remark}{section}
\newcommand{\noi}{\noindent}

\newcommand{\dsty}{\displaystyle}
\newcommand{\txty}{\textstyle}

\newcommand{\al}{\alpha}

\newcommand{\gm}{\gamma}
\newcommand{\dl}{\delta}

\newcommand{\lm}{\lambda}

\newcommand{\varep}{\varepsilon}
\newcommand{\eps}{\epsilon}
\newcommand{\vp}{\varphi}
\newcommand{\sig}{\sigma}

\newcommand{\om}{\omega}
\newcommand{\Om}{\Omega}

\newcommand{\z}{\zeta}

\newcommand{\df}[1]{\buildrel\mbox{\small def}\over{#1}}

\newcommand{\nn}{\mathbb{N}}

\newcommand{\rr}{\mathbb{R}}
\newcommand{\rn}{\rr^N}

\newcommand{\bl}[1]{\mathbf{#1}}

\newcommand{\os}{\vbox{\hrule \hbox{\vrule 
height.6em depth0pt 
\hskip.6em \vrule height.6em depth0em}
\hrule}} 

\newcommand{\dvg}{\operatorname{div}}

\newcommand{\osc}{\operatornamewithlimits{osc}}

\newcommand{\loc}{\operatorname{loc}}

\newcommand{\pl}{\partial}

\newcommand{\intl}{\int\limits}

\def\Xint#1{\mathchoice
    {\XXint\displaystyle\textstyle{#1}}%
    {\XXint\textstyle\scriptstyle{#1}}%
    {\XXint\scriptstyle\scriptscriptstyle{#1}}%
    {\XXint\scriptscriptstyle\scriptscriptstyle{#1}}%
    \!\int}
\def\XXint#1#2#3{\setbox0=\hbox{$#1{#2#3}{\int}$}
    \vcenter{\hbox{$#2#3$}}\kern-0.5\wd0}
\def\bint{\Xint-}
\def\dashint{\Xint{\raise4pt\hbox to7pt{\hrulefill}}}
\def\dashiint{\bint\kern-0.15cm\bint}

\newcommand{\ovl}[3]{\int_{#1}^{#2}\kern-#3pt\raise4pt\hbox to7pt{\hrulefill}\ }

\newcommand{\ovll}[3]{\intl_{#1}^{#2}\kern-#3pt\raise4pt\hbox to7pt{\hrulefill}\ }

\newcommand{\tvl}[2]{\iint_{#1}\kern-#2pt\raise4pt\hbox to7pt{\hrulefill}\ }

\newcommand{\bye}{\end{document}}

\newcommand{\datap}{\{p,N,C_o,C_1\}}

\newcommand{\pto}{(x_o,t_o)}

\newcommand{\tvls}[2]{\iint_{#1}\kern-#2pt\raise4pt\hbox to15pt{\hrulefill}\ }

\newcommand{\ibtr}{\intl_0^t\!\intl_{K_{2\rho}}\kern-4pt}

\begin{document}
\title{A Boundary Estimate for Degenerate Parabolic Diffusion Equations}
\author
{Ugo Gianazza\\
Dipartimento di Matematica ``F. Casorati", 
Universit\`a di Pavia\\ 
via Ferrata 1, 27100 Pavia, Italy\\
email: {\tt gianazza@imati.cnr.it}
\and
Naian Liao\thanks{Corresponding author}\\
College of Mathematics and Statistics\\
Chongqing University\\
Chongqing, China, 401331\\
email: {\tt liaon@cqu.edu.cn}}
\date{}
\maketitle
\vskip.4truecm
\begin{abstract}
We prove an estimate on the modulus of continuity at a boundary
point of a cylindrical domain for local weak solutions to degenerate
parabolic equations of $p$-laplacian type. The estimate is given in
terms of a Wiener-type integral, defined by a proper elliptic $p$-capacity.
\end{abstract}

\section{Introduction}\label{S:intro}
Let $E$ be an open set in $\rn$ and for $T>0$ let $E_T$ denote the
cylindrical domain $E\times(0,T]$. Moreover let
\begin{equation*}
S_T=\partial E\times(0,T],\qquad \partial_p E_T=S_T\cup(\bar{E}\times\{0\})
\end{equation*}
denote the lateral, and the parabolic boundary respectively.

We shall consider quasi-linear, parabolic partial differential equations of the form
\begin{equation}  \label{Eq:1:1}
u_t-\dvg\bl{A}(x,t,u, Du) = 0\quad \text{ weakly in }\> E_T,
\end{equation}
where the function $\bl{A}:E_T\times\rr^{N+1}\to\rn$ is only assumed to be
measurable and subject to the structure conditions
\begin{equation}  \label{Eq:1:2}
\left\{
\begin{array}{l}
\bl{A}(x,t,u,\xi)\cdot \xi\ge C_o|\xi|^p \\
|\bl{A}(x,t,u,\xi)|\le C_1|\xi|^{p-1}%
\end{array}%
\right .\quad \text{ a.e.}\> (x,t)\in E_T,\, \forall\,u\in\rr,\,\forall\xi\in\rn,
\end{equation}
where $C_o$ and $C_1$ are given positive constants, and $p>2$.

We refer to the parameters $\datap$ as our structural data, and we write $\gm%
=\gm(p,N,C_o,C_1)$ if $\gm$ can be quantitatively
determined a priori only in terms of the above quantities.
A function
\begin{equation}  \label{Eq:1:4}
u\in C\big(0,T;L^2_{\loc}(E)\big)\cap L^p_{\loc}\big(0,T; W^{1,p}_{%
\loc}(E)\big)
\end{equation}
is a local, weak sub(super)-solution to \eqref{Eq:1:1}--\eqref{Eq:1:2} if
for every compact set $K\subset E$ and every sub-interval $[t_1,t_2]\subset
(0,T]$
\begin{equation}  \label{Eq:1:5}
\int_K u\vp dx\bigg|_{t_1}^{t_2}+\int_{t_1}^{t_2}\int_K \big[-u\vp_t+\bl{A}%
(x,t,u,Du)\cdot D\vp\big]dxdt\le(\ge)0
\end{equation}
for all non-negative test functions
\begin{equation*}
\vp\in W^{1,2}_{\loc}\big(0,T;L^2(K)\big)\cap L^p_{\loc}\big(0,T;W_o^{1,p}(K)%
\big).
\end{equation*}
This guarantees that all the integrals in \eqref{Eq:1:5} are convergent.

For any $k\in\rr$, let
\[
(v-k)_-=\max\{-(v-k),0\},\qquad(v-k)_+=\max\{v-k,0\}.
\]
We require \eqref{Eq:1:1}--\eqref{Eq:1:2} to be parabolic, namely that whenever $u$ is a weak solution, for all $k\in\rr$, the functions $(u-k)_\pm$ are weak sub-solutions, with $\bl{A}(x,t,u,Du)$ replaced by $\pm\bl{A}(x,t,k\pm(u-k)_\pm,\pm D(u-k)_\pm)$. As discussed in condition ({\bf A}$_6$) of \cite[Chapter II]{dibe-sv} or Lemma~1.1 of \cite[Chapter 3]{DBGV-mono}, such a condition is satisfied, if for all {$(x,t,u)\in E_T\times\rr$} we have
\[
\bl{A}(x,t,u,\eta)\cdot\eta\ge0\qquad\forall\,\eta\in\rn,
\]
which is guaranteed by \eqref{Eq:1:2}.


For $y\in \rn$ and $\rho > 0 $, $K_{\rho}(y)$ denotes the cube of edge $2\rho$, centered at $y$
with faces parallel to the coordinate planes. When $y$ is the
origin of $\rn$, we simply write $K_{\rho}$.

We are interested in the boundary behaviour of solutions to the Cauchy-Dirichlet problem
\begin{equation}\label{Eq:1:6}
\left\{
\begin{aligned}
&u_t-\dvg\bl{A}(x,t,u, Du) = 0\quad \text{ weakly in }\> E_T\\
&u(\cdot,t)\Big|_{\partial E}=g(\cdot,t)\quad \text{ a.e. }\ t\in(0,T]\\
&u(\cdot,0)=g(x,0),
\end{aligned}
\right.
\end{equation}
where 
\begin{itemize}
\item {\bf (H1)}: $\bl{A}$ satisfies \eqref{Eq:1:2} for $p>2$, as already mentioned before;
\item {\bf (H2)}: $\dsty g\in L^p(0,T;W^{1,p}( E))$, and $g$ is
  continuous on $\overline{E}_T$ with modulus of continuity
  $\om_g(\cdot)$.
\end{itemize}
We do not impose any {\it a priori} requirements on the boundary of the domain
$E\subset\rn$.


A weak sub(super)-solution to
the Cauchy-Dirichlet problem \eqref{Eq:1:6} is a measurable function $u\in C\big(0,T;L^2(E)\big)\cap 
L^p\big(0,T; W^{1,p}(E)\big)$ satisfying
\begin{equation} \label{Eq:1:7}
\begin{aligned}
&\int_E u\vp(x,t) dx+\iint_{E_T} \big[-u\vp_t+\bl{A}%
(x,t,u,Du)\cdot D\vp\big]dxdt\\
&\le(\ge)\int_E g\vp(x,0) dx
\end{aligned}
\end{equation}
for all non--negative test functions
\begin{equation*}
\vp\in W^{1,2}\big(0,T;L^2(E)\big)\cap L^p\big(0,T;W_o^{1,p}(E)\big).
\end{equation*}
In addition, we take the boundary condition $u\le g$ ($u\ge g$) to
mean that $(u-g)_+(\cdot,t)\in W^{1,p}_o(E)$ ($(u-g)_-(\cdot,t)\in W^{1,p}_o(E)$) for
a.e. $t\in(0,T]$. A function $u$ which is both a weak sub-solution and
a weak super-solution, is a solution. Notice that the range we are
assuming for $p$, and the continuity of $g$ on the closure of $E_T$
ensure that a weak solution $u$ to \eqref{Eq:1:6} is bounded (see, for example, \cite[Chapter~V, Theorem~3.3]{dibe-sv}).

Let $\pto\in S_T$; the relative capacity of $E^c$ at $x_o$ is defined as 
\begin{equation}\label{Eq:delta}
\dl(\rho)\df=\frac{{\rm cap}_p(K_{\rho}(x_o)\backslash E,K_{\frac32\rho}(x_o))}{{\rm cap}_p(K_{\rho}(x_o),K_{\frac32\rho}(x_o))}.
\end{equation}
We refer to Section~\ref{Sec:prelim} for more details on the notion of capacity. 
In the sequel, we always assume $x_o$ is a {\it Wiener point of the domain $E$}, i.e.,
\begin{equation}\label{Eq:fat}
\int_0^1[\dl(s)]^{\frac{1}{p-1}}\frac{ds}{s}=\infty.
\end{equation}
Let $\gamma_*>1$ be the constant claimed in Lemma~\ref{Lm:3:3}; fix $R_o>0$ and $0<\eps<1$, such that 
 \begin{equation}\label{Eq:eps-ro}
 (t_o-3\gamma_*[\dl(R_o)]^{\frac{2-p}{p-1}}R_o^{p-\eps},t_o]\subset(0,T],
 \end{equation}
and set
\[
Q_{R_o}=K_{2R_o}(x_o)\times(t_o-3\gamma_*[\dl(R_o)]^{\frac{2-p}{p-1}}R_o^{p-\eps},t_o].
\]
Condition \eqref{Eq:eps-ro} can always be realized, since otherwise we would have for all $s\in(0,1)$ that
 \[
3\gamma_* [\dl(s)]^{\frac{2-p}{p-1}}s^{p-\eps}\ge t_o,
 \]
and consequently
 \[
\int_0^1[\dl(s)]^{\frac{1}{p-1}}\frac{ds}{s}\le \bigg(\frac{3\gamma_*}{t_o}\bigg)^{\frac{1}{p-2}}\int_{0}^1 s^{\frac{2-\eps}{p-2}}\,ds
=\bigg(\frac{3\gamma_*}{t_o}\bigg)^{\frac{1}{p-2}}\frac{p-2}{p-\eps}<\infty.
 \]
We can now state the main result of this work.
\begin{theorem}\label{Thm:1:1}
Let $u$ be a weak solution to \eqref{Eq:1:6}, assume that {\bf (H1)}--{\bf (H2)} and
\eqref{Eq:fat} are satisfied, choose $R_o$ and $\eps$ such that \eqref{Eq:eps-ro} holds true.
Then there exist positive constants $\gm\in(0,1)$, and $\bar\gm>0$ that depend only on the data $\datap$,
such that for any $\rho\in(0,R_o)$
\begin{equation}\label{Eq:decay}
\osc_{Q_{\rho}(\om_o)\cap E_T}\,u\le\om_o\exp\left\{-\gm\int_{\rho}^{R_o}[\dl(s)]^{\frac1{p-1}}\frac{ds}s\right\}+\osc_{{Q}_{R_o}\cap S_T}g+\bar\gm R_o^{\frac{\eps}{p-2}},
\end{equation}
where $\dl(s)$ is defined in \eqref{Eq:delta}, and
\[ 
\om_o\df=\osc_{Q_{R_o}}u,\qquad Q_{\rho}(\om_o)=K_{2\rho}(x_o)\times[t_o-\om_o^{2-p}\rho^p, t_o].
\]
\end{theorem}
By the same argument of proving \eqref{Eq:eps-ro}, one easily obtains that 
there is a sequence of positive numbers $\{R_n\}$ converging to zero, such that
\[
3\gamma_*[\dl(R_n)]^{\frac{2-p}{p-1}}R_n^{p-\eps}\to 0\quad\text{ as }n\to\infty.
\]
Therefore, from Theorem~\ref{Thm:1:1} we can conclude
  the following corollary in a standard way.
\begin{corollary}\label{Cor:1:1}
Let $u$ be a weak solution to \eqref{Eq:1:6}, assume that {\bf (H1)}--{\bf (H2)} hold true,  
that $(x_o,t_o)\in S_T$, and that $x_o$ is a Wiener point of the domain $E$. Then 
\[
\lim_{\genfrac{}{}{0pt}{}{(x,t)\to(x_o,t_o)}{(x,t)\in E_T}}u(x,t)=g(x_o,t_o).
\]
\end{corollary}

As already remarked in \cite{GLL}, Theorem \ref{Thm:1:1} also implies H\"older regularity up to the
boundary under a fairly weak assumption on the domain. More
specifically, a set $A\subset\rn$ is {\it uniformly $p$-fat}, if for some
$\gm_o,\,\rho_o>0$ one has
\[
\frac{{\rm cap}_p(K_{\rho}(x_o)\cap A,K_{\frac32\rho}(x_o))}{{\rm
    cap}_p(K_{\rho}(x_o),K_{\frac32\rho}(x_o))}\geq \gm_o
\]
for all $0<\rho <\rho_o$ and all $x_o\in A$. See \cite{lewis1988} for
more on this notion. We have the following corollary. 

\begin{corollary}\label{cor:holder}
  Let $u$ be a weak solution to \eqref{Eq:1:6}, assume that {\bf (H1)}--{\bf (H2)} hold true, the complement of the domain $E$ is
  uniformly $p$-fat, and let $g$ be H\"older continuous.  Then the
  solution $u$ is H\"older continuous up to the boundary.
\end{corollary}

\begin{remark}\label{Rmk:1:1}
{\normalfont When $p>N$, then for any $s\in(0,1)$ we always have $\dl(s)\ge\gm_o$ for some $\gm_o\in(0,1)$ depending only
on $N$ and $p$, as explained in Section~\ref{Sec:prelim}.
In such a case, if $g$ is assumed to be H\"older continuous, then 
Corollary~\ref{cor:holder} is automatically satisfied, and
\begin{equation}\label{Eq:hol-decay}
\osc_{Q_{\rho}(\om_o)\cap E_T}\,u\le\om_o \left(\frac{\rho}{ R_o}\right)^{\al},
\end{equation}
where $\al\in(0,1)$ depends only on the data $\datap$.
}
\end{remark}
\subsection{Novelty and Significance}
The continuity at the boundary of rough sets for solutions to elliptic partial differential equations of $p$-laplacian type is by now basically a settled matter (see, for example, \cite{Ma-Zie}). In the parabolic setting the theory is more fragmented, and still to be fully developed. 

Continuity at the boundary for quite general operators with a growth of order $p=2$ has been considered in \cite{ziemer, ziemer1982}. When dealing with a general $p>1$, the fact that a Wiener point is a continuity point has already been observed in \cite{BBGP} (see also \cite{KiLi96}). However, only the prototype parabolic $p$-laplacian is dealt with, and no explicit decay estimate as in \eqref{Eq:decay} is provided.

The so-called super-critical singular range, that is when $\frac{2N}{N+1}<p<2$, has been considered in \cite{skrypnik-2004} based on the comparison principle, 
and then, more recently in \cite {GLL}, with different techniques, which are closely related to the method we use here. Coming to the degenerate range $p>2$, a result similar to ours is stated in \cite{skrypnik-2000}. In such a paper, 
the comparison principle once more plays a fundamental role; this is not the case here, where no use whatsoever of the comparison principle is made, and purely structural estimates are proved.
Moreover, we give an explicit modulus of continuity, and
therefore, Theorem~\ref{Thm:1:1} represents a step forward. 

Here we also point out a difference between the singular case and the degenerate case,
when proving the reduction of oscillation along a family of nested, intrinsically scaled cylinders.
In the singular case, we do not require {\it a priori} that the Wiener integral \eqref{Eq:fat} diverges.
However, in the degenerate case, we need to use the divergence of the Wiener integral in order to
fit the cylinders in one another, due to the role played by $\dl(\rho)$
in the time scaling (see Lemma~\ref{Lm:4:1}).

As already remarked in \cite{GLL} for an analogous result, Corollary~\ref{cor:holder} can be seen as an extension of Theorem~1.2 of \cite[Chapter III]{dibe-sv}, where the H\"older continuity up to the boundary of weak solutions to the Cauchy-Dirichlet problem \eqref{Eq:1:6} with H\"older continuous boundary data is proved, assuming that the domain $E$ satisfies a positive geometric density condition. It is a matter of straightforward computations to see that if a domain $E$ has positive geometric
density, then the complement of $E$ is uniformly $p$-fat, but the
opposite implication obviously does not hold.

As pointed out in  Remark~\ref{Rmk:1:1}, when $p>N$, and the boundary datum is H\"older continuous, the solution is also H\"older continuous, regardless of the geometry of the domain $E$. This is obvious for the elliptic $p$-laplacian due to the Sobolev embedding, but the parabolic case seems new.

Finally, all the estimates are stable as $p\to2+$, and therefore, the continuity result of Corollary~\ref{Cor:1:1}
recovers the analogous one given in  \cite{ziemer}.

As for the structure of the paper, the proof of Theorem~\ref{Thm:1:1} is given in Section~\ref{S:final}, whereas the previous sections are devoted to introductory material, namely some preliminary results (Section~\ref{Sec:prelim}), and a couple of auxiliary lemmas (Section~3).

\begin{ack} 
{\normalfont Part of this paper was written, when Ugo Gianazza visited the College of Mathematics and Statistics of
 Chongqing University, and it was finished during the  {\it Workshop on Nonlinear Parabolic PDEs} at Institut Mittag-Leffler, June 2018. The warm hospitality of both institutions is gratefully acknowledged. 
 The authors are indebted to
 Emmanuele DiBene\-det\-to and Vincenzo Vespri for their comments, 
 which greatly helped to improve the final version of this manuscript.}
\end{ack}
\section{Preliminaries}\label{Sec:prelim}
The first basic fact is taken from \cite[Lemma 2.2]{kuusi2008} (see also \cite[Lemma 10.1 on page 116]{DBGV-mono}).
\begin{lemma}\label{Lm:2:1}
Let $u$ be a non-negative, local, weak super-solution to the degenerate
equation \eqref{Eq:1:1}--\eqref{Eq:1:2} in the cylinder
\[
K\times(t_1,t_2)
\]
where $K$ is a cube in $\rr^N$. Then for all $\varep\in(-1,0)$,
\begin{equation}
\begin{aligned}
\frac{p}{C_o(1+\varep)|\varep|}&\sup_{t_1<t<t_2}\int_{K}u^{1+\varep}\vp^p(x,t)\,dx
+\int_{t_1}^{t_2}\int_K|Du|^pu^{\varep-1}\vp^p\,dxdt\\
&\le \bigg(\frac{C_1p}{C_o|\varep|}\bigg)^p\int_{t_1}^{t_2}\int_K u^{\varep+p-1}|D\vp|^p\,dxdt\\
&\quad+\frac{p}{C_o(1+\varep)|\varep|}\int_{t_1}^{t_2}\int_K u^{1+\varep}
\bigg(\frac{\pl \vp^p}{\pl t}\bigg)_+\,dxdt\\
&\quad+\frac{p}{C_o(1+\varep)|\varep|}\int_{K}u^{1+\varep}\vp^p(x, t_2)\,dx
\end{aligned}
\end{equation}
for every non-negative test function
\[
\vp\in W^{1,2}(t_1,t_2;L^2(K))\cap L^p(t_1,t_2;W^{1,p}_o(K)).
\]
\end{lemma}
\begin{proof}
Take the test function $(u+\nu)^{\varep}\vp^p$ in the weak formulation \eqref{Eq:1:4}
where $\nu$ is a positive constant. Then a routine calculation followed by
letting $\nu\to 0$ yields the conclusion.
\end{proof}
With the above lemma at disposal, we are able to show the following reverse 
H\"older's inequality. This is done by carefully tracing the dependence in the proof
of \cite[Lemma 5.3]{kuusi2008} or \cite[Lemma 11.1]{DBGV-mono}.
\begin{lemma}\label{Lm:2:2}
Let $v$ be a non-negative, local, weak super-solution to the degenerate
equation \eqref{Eq:1:1}--\eqref{Eq:1:2} in the cylinder
\[
K_{2\rho}(x_o)\times(t_o-\theta\rho^p, t_o),
\]
with $\theta>0$ to be determined later. For any $\sig\in(0,1)$, and for any $\eta\in(0,1)$, there exists a constant $C_\eta>1$ depending only
on the data $\datap$, $\sig$, and $\eta$, such that
\[
\begin{aligned}
\dashint_{t_o-\theta\rho^p}^{t_o}&\dashint_{K_{\rho}(x_o)}v^{p-2+\sig(1+\frac{p}{N})}\,dxdt\\
\le& C_\eta\bigg[\sup_{t_o-\theta\rho^p<t<t_o}\dashint_{K_{2\rho}(x_o)}v(x,t)\,dx\bigg]^{p-2+\sig(1+\frac{p}{N})}+\eta\left(\frac1\theta\right)^{1+\frac\sig{p-2}\left(1+\frac p N\right)}.
\end{aligned}
\]
\end{lemma}
\begin{proof}
By a change of variables, we may consider this problem in the cylinder
\[
Q=K_1\times(-\theta, 0]. 
\]
Furthermore, for $i=1,\,2$ let us set 
\[
Q_{r_i}=K_{r_i}\times(-\theta,0]\quad\text{with}\quad\frac12<r_1<r_2<1;
\]
pick a non-negative, piecewise smooth, cutoff function on $K_{r_2}$, such that
\[
0\le\vp\le1,\qquad\vp=1\quad\text{ in }K_{r_1},\qquad |D\vp|\le\frac{1}{r_2-r_1}.
\]
An application of the parabolic Sobolev embedding (see, for example, \cite[Chapter~I, Proposition~3.1]{dibe-sv})
gives us that
\begin{align*}
\iint_{Q_{r_1}}&v^{p-2+\sig(1+\frac{p}{N})}\,dxdt\\
&\le\gm\iint_{Q_{r_2}}|D(v^{\frac{p-2+\sig}{p}}\vp)|^p\,dxdt\,
\bigg(\sup_{-\theta<t<0}\int_{K_{1}}v^{\sig}(x,t)\,dx\bigg)^{\frac{p}{N}}.
\end{align*}
For simplicity, let 
\[
M_\sig\df=\sup_{-\theta<t<0}\int_{K_{1}}v^{\sig}(x,t)\,dx.
\]
By Lemma \ref{Lm:2:1} with $\varep=-1+\sig$ we have
\begin{align*}
\iint_{Q_{r_2}}&|D(v^{\frac{p-2+\sig}{p}}\vp)|^p\,dxdt\\
&\le\gm\bigg(M_{\sig}+\iint_{Q_{r_2}}v^{p-2+\sig}|D\vp|^p\,dxdt\bigg).
\end{align*}
By Young's inequality
\begin{align*}
\iint_{Q_{r_2}}v^{p-2+\sig}|D\vp|^p M_{\sig}^{\frac{p}{N}}\,dxdt
&\le \frac1{9\gm}\iint_{Q_{r_2}}v^{p-2+\sig(1+\frac{p}{N})}\,dxdt\\
&+\gm M_{\sig}^{\frac{p-2+\sig(1+\frac{p}{N})}{\sig}}\iint_{Q_{r_2}}|D\vp|^{p+\frac{N(p-2+\sig)}{\sig}}\,dxdt.
\end{align*}
Combine the above estimates to obtain that
\begin{align*}
\iint_{Q_{r_1}}v^{p-2+\sig(1+\frac{p}{N})}\,dxdt&\le \frac19\iint_{Q_{r_2}}v^{p-2+\sig(1+\frac{p}{N})}\,dxdt\\
&+\gm M_{\sig}^{1+\frac{p}{N}}+
\gm\theta M_{\sig}^{\frac{p-2}{\sig}+1+\frac{p}{N}}\bigg(\frac{1}{r_2-r_1}\bigg)^{p+\frac{N(p-2+\sig)}{\sig}}.
\end{align*}
By an interpolation argument (see, for example, \cite[Chapter~I, Lemma~4.3]{dibe-sv}) one arrives at
\[
\begin{aligned}
\int_{-\theta}^{0}\int_{K_{\frac12}}v^{p-2+\sig(1+\frac{p}{N})}\,dxdt&\le\gm M_{\sig}^{1+\frac{p}{N}}+
\gm\theta M_{\sig}^{\frac{p-2}{\sig}+1+\frac{p}{N}},\\
\dashint_{-\theta}^{0}\int_{K_{\frac12}}v^{p-2+\sig(1+\frac{p}{N})}\,dxdt&\le\gm \frac{M_{\sig}^{1+\frac{p}{N}}}\theta+
\gm M_{\sig}^{\frac{p-2}{\sig}+1+\frac{p}{N}}.
\end{aligned}
\]
Note that
\[
\begin{aligned}
\gm M_{\sig}^{\frac{p-2}{\sig}+1+\frac{p}{N}}&=\gm \bigg(\sup_{-\theta<t<0}\int_{K_{1}}v^{\sig}(x,t)\,dx\bigg)^{\frac{p-2}{\sig}+1+\frac{p}{N}}\\
&\le\gm\bigg(\sup_{-\theta<t<0}\int_{K_{1}}v(x,t)\,dx\bigg)^{{p-2}+\sig\left(1+\frac{p}{N}\right)},
\end{aligned}
\]
and
\[
\begin{aligned}
\frac\gm\theta M_{\sig}^{1+\frac pN}&=\frac\gm\theta\bigg(\sup_{-\theta<t<0}\int_{K_1} v^\sig(x,t)\,dx\bigg)^{1+\frac pN}\\
&\le\frac\gm\theta\bigg(\sup_{-\theta<t<0}\int_{K_1} v(x,t)\,dx\bigg)^{\sig\left(1+\frac pN\right)}\\
&\le C_\eta\bigg(\sup_{-\theta<t<0}\int_{K_1} v(x,t)\,dx\bigg)^{p-2+\sig\left(1+\frac pN\right)}+\eta\left(\frac1\theta\right)^{1+\frac\sig{p-2}\left(1+\frac pN\right)};
\end{aligned}
\]
therefore, we conclude that
\[
\begin{aligned}
\dashint_{-\theta}^0\int_{K_{\frac12}} v^{p-2+\sig(1+\frac{p}{N})}\,dxdt
\le&C_\eta\bigg(\sup_{-\theta<t<0}\int_{K_1} v(x,t)\,dx\bigg)^{p-2+\sig(1+\frac{p}{N})}\\
&+\eta\left(\frac1\theta\right)^{1+\frac\sig{p-2}\left(1+\frac pN\right)}.
\end{aligned}
\]
Returning to the original variables yields the desired result.
\end{proof}
\begin{remark}
{\normalfont If one chooses
\[
\theta=\bigg[\dashint_{K_{2\rho}(x_o)}v(x,t_o)\,dx\bigg]^{2-p},
\]
then
\[
\dashint_{t_o-\theta\rho^p}^{t_o}\dashint_{K_{\rho}(x_o)}v^{p-2+\sig(1+\frac{p}{N})}\,dxdt
\le\gm\bigg[\sup_{t_o-\theta\rho^p<t<t_o}\dashint_{K_{2\rho}(x_o)}v\,dxdt\bigg]^{p-2+\sig(1+\frac{p}{N})}.
\]
}
\end{remark}
\vskip.2truecm
\noi We will need the following weak Harnack inequality, proved in \cite{kuusi2008}.
\begin{theorem}\label{Thm:3:7:1}
Let $u$ be a non-negative, local, weak super-solution 
to \eqref{Eq:1:1}--\eqref{Eq:1:2}.  There exist positive 
constants $c$ and $\gm_o$, depending only on the data 
$\datap$, such that for a.e. $s\in(0,T)$
\begin{equation}\label{Eq:3:7:1}
\bint_{K_\rho(y)}u(x,s)dx\le c
\Big(\frac{\rho^p}{T-s}\Big)^{\frac1{p-2}}
+\gm_o\inf_{K_{4\rho}(y)}u(\cdot,t)
\end{equation}
for all times 
\begin{equation*}
s+{\txty\frac12}\theta\rho^p\le t\le s+\theta\rho^p
\end{equation*}
where 
\begin{equation}\label{Eq:3:7:2}
\theta=\min\Big\{c^{2-p}\frac{T-s}{\rho^p}\,,\,\Big(
\bint_{K_\rho(y)}u(x,s)dx\Big)^{2-p}\Big\}.
\end{equation}
\end{theorem}
\begin{remark}
{\normalfont If $s$ and $\rho$ are chosen such that
\[
s+\frac {2c^{p-2}}{\left[\dsty\bint_{K_\rho(y)}u(x,s)\,dx\right]^{p-2}}\,\rho^p<T,
\]
then
\begin{align*}
&c\left(\frac{\rho^p}{T-s}\right)^{\frac1{p-2}}<\frac1{2^{\frac1{p-2}}}\bint_{K_\rho(y)}u(x,s)\,dx\\
&\theta=\left[\bint_{K_\rho(y)}u(x,s)\,dx\right]^{2-p},
\end{align*}
and therefore, 
\begin{equation}\label{WHI}
\bint_{K_\rho(y)}u(x,s)dx\le\bar\gm\inf_{K_{4\rho}(y)}u(\cdot,t)
\end{equation}
for all times 
\begin{equation*}
s+{\txty\frac12}\theta\rho^p\le t\le s+\theta\rho^p.
\end{equation*}
Moreover, $\bar\gm=\frac{\gm_o}{1-\left(\frac12\right)^{\frac1{p-2}}}$, and therefore the constant is stable as $p\to2$.}
\end{remark}
\vskip.2truecm
\noi Another result we will rely on is the following (see \cite[Corollary~3.1]{GSV}).
\begin{lemma}\label{LBL1}
Let $u$ be a non-negative, local, weak super-solution to \eqref{Eq:1:1}--\eqref{Eq:1:2} 
in the cylinder $K_{2\rho}(y)\times [\bar t,\bar t+T]$. 
Suppose that $$\inf_{K_{2\rho}(y)} u(x,\bar t)\geq k\quad\text{ for some }k>0.$$ 
Then for all $t\in (\bar t,\bar t+T]$ we have
\begin{equation}\label{LB}
\inf_{K_\rho(y)} u(x,t)\geq \frac{k}{2}\left( 1+\frac{t-\bar t}{\nu k^{2-p} (2\rho)^p}\right)^\frac{1}{2-p},
\end{equation}
where $\nu\in(0,1)$ is a constant that depends only on the data $\datap$. 
\end{lemma}
Finally, we recall the notion of capacity introduced in \cite[\S~4]{GLL}.

Let $\Om\subset\rn$ be an open set, and $Q\df=\Om\times(t_1,t_2)$: $Q$ is an open cylinder in $\rr^{N+1}$. In the following we will refer to such sets as {\it open parabolic cylinders}. For any compact set $K\subset Q$, we define the {\it parabolic capacity of $K$ with respect to $Q$} as 
\begin{equation}\label{Eq:4:1}
\begin{aligned}
\gm_p(K,Q)=&\inf\left\{\iint_Q|D\vp|^p\,dxdt: \right.\\
&\vp\in C^\infty_o(Q),\ \ \vp\ge1\ \ \text{ on a neighborhood of }\ \ K\bigg\},
\end{aligned}
\end{equation}
where $D\vp$ denotes the gradient of $\vp$ with respect to the space variables only. 

The notion of the elliptic capacity is quite standard. Indeed, for every compact set $F\subset \Om$ we define
\[
{\rm cap}_p(F,\Om)=\inf\left\{\int_{\Om}|D\psi|^p\,dx: \psi\in C_o^{\infty}(\Om),\ \ \psi\ge1\text{ on }F\right\}.
\] 
For $p\ge N$ one always assume that $\Omega\subset\subset\rn$, since ${\rm cap}_p(F,\rn)=0$ for $p\ge N$.
It should be remarked that an explicit calculation (see, for example, \cite[page 35]{HKM}) gives us that
\begin{equation*}
\begin{aligned}
&{\rm cap}_p(K_\rho(x_o), K_{2\rho}(x_o))=c_1(N,p)\rho^{N-p},\quad \forall p>1,\\
&{\rm cap}_p(\{x_o\}, K_{\rho}(x_o))=c_2(N,p)\rho^{N-p},\quad \forall p>N,
\end{aligned}
\end{equation*}
where $c_1$ and $c_2$ are positive constants with indicated dependence.
Hence, when $p>N$, the relative capacity defined in \eqref{Eq:delta} is always
bounded below, i.e.,
\[
\dl(\rho)\ge\gm_o(N,p)\quad\text{for some }\gm_o\in(0,1).
\]
As a result, condition \eqref{Eq:fat} always holds when $p>N$. For further details and properties about the elliptic capacity, see for example
\cite[Chapter~4]{Evans-Gariepy}, \cite[Chapter~2]{HKM}, \cite[Chapter~2]{Ma-Zie}, or \cite{frehse}.

Now we point out the connection between the two notions of capacity. 
Let $Q=\Omega\times(t_1,t_2)$, and for any set $E\subset\rr^{N+1}$ define $E_\tau=E\cap\{t=\tau\}$. Then, we have the following result (see \cite[Proposition~A.2]{Biroli-Mosco}).
\begin{proposition}\label{Prop:4:2}
Let $K\subset Q$ be compact. Then,
\begin{equation}\label{eq:4:2}
\gm_p(K,Q)=\int_{t_1}^{t_2}{\rm cap}_p(K_\tau,\Omega)\,d\tau.
\end{equation}
\end{proposition}
\section{Auxiliary Lemmas}
Fix $(x_o,t_o)\in S_T$, and consider the cylinder 
\begin{equation}\label{Eq:cyl1}
Q=K_{16\rho}(x_o)\times[s,t],
\end{equation} 
where $s$, $t$ are such that $0<s<t_o<t<T$, and let $\Sigma\df=S_T\cap Q$. 
Our estimates are based on the following lemma, stated and proved in \cite[Lemma~2.1]{GLL}.

\begin{lemma}\label{Lm:2:0}
 Take any number $k$ such that
 \begin{equation}\label{Eq:level-k}
 k\ge\sup_\Sigma g.
 \end{equation} 
 Let $u$ be a weak solution to \eqref{Eq:1:6} in the cylinder $Q$, and define 
  \begin{displaymath}
    u_k=
    \begin{cases}
      (u-k)_+, &\text{in }Q\cap E_T, \\
      0, & \text{in }Q\setminus E_T.
    \end{cases}
  \end{displaymath}  
  Then $u_k$ is a (local) weak sub-solution to \eqref{Eq:1:6} in the cylinder $Q$. The
  same conclusion holds for the zero extension of $u_h=(h-u)_+$ for
  truncation levels $h\leq \inf_{\Sigma}g$.
\end{lemma}

Let $k$ be any number which satisfies \eqref{Eq:level-k}, and
\begin{equation}\label{Eq:super}
\left\{
\begin{aligned}
&\text{define}\ u_k=(u-k)_+,\\
&\text{choose}\ \mu>0\ \text{ such that }\ \mu\ge\sup_Q u_k,\\ 
&\text{define}\ v:Q\rightarrow\rr_+,\quad v=\mu-u_k.
\end{aligned}
\right.
\end{equation}
It is not hard to verify that $v$ is a weak super-solution to
\eqref{Eq:1:6} in the whole $Q$. 

Finally, let
\begin{align*}
&\dl(\rho)\df=\frac{{\rm cap}_p(K_{\rho}(x_o)\backslash E,K_{\frac32\rho}(x_o))}{{\rm cap}_p(K_{\rho}(x_o),K_{\frac32\rho}(x_o))},
\end{align*}
%
and
\begin{equation}\label{Eq:theta-bar}
\bar\theta\df=\left(\mu [\dl(\rho)]^{\frac{1}{p-1}}\right)^{2-p}.
\end{equation}
We have the following.
\begin{lemma}\label{Lm:5:1}
Let $\pto$, $Q$, $u_k$, $\mu$, $v$ as in \eqref{Eq:cyl1}--\eqref{Eq:super}, take $\bar\theta$ as in \eqref{Eq:theta-bar}, and assume that $t_o-\bar\theta\rho^p\ge s$. Then there exists a constant $\gm_1>1$, that depends only on the data $\datap$, such that
\begin{equation}\label{Eq:low-bd2}
\mu [\dl(\rho)]^{\frac{1}{p-1}}\le
\gm_1\sup_{t_o-\bar\theta\rho^p<t<t_o}\dashint_{K_{2\rho}(x_o)} v(x,t)\,dx.
\end{equation}
\end{lemma}
\begin{proof}
Without loss of generality, we may assume that $(x_o,t_o)=(0,0)$. Construct three cylinders:
\begin{align*}
&Q_1=K_{\rho}\times(-\frac34\bar\theta\rho^p,-\frac14\bar\theta\rho^p);\\
&Q_2=K_{\frac32\rho}\times(-\frac78\bar\theta\rho^p,-\frac18\bar\theta\rho^p);\\
&Q_3=K_{2\rho}\times(-\bar\theta\rho^p,0).
\end{align*}
Introduce the standard cut-off functions $\z$ and $\vp$ such that
\begin{equation*}
\z(x,t)=\left\{
\begin{array}{ll}
{1}\quad (x,t)\in Q_1\\
{}\\
{0}\quad (x,t)\notin Q_2
\end{array}\right.
\text{and}\quad
\vp(x,t)=\left\{
\begin{array}{ll}
{1}\quad (x,t)\in Q_2\\
{}\\
{0}\quad (x,t)\notin Q_3.
\end{array}\right.
\end{equation*}
We use the test function $u_k\z^p$ in the weak formulation,
modulus a standard Steklov average;
a straightforward calculation similar to the one in \cite[Lemma 5.1]{GLL} gives us that
\begin{equation*}
\begin{aligned}
\iint_{Q_2}|D(v\z)^p|\,dxdt&\le\gm\mu\iint_{Q_2}|Dv|^{p-1}|D\z|\,dxdt+\gm\iint_{Q_2}v^p|D\z|^p\,dxdt\\
&+\gm\mu\iint_{Q_2}v|\z_t|\,dxdt.
\end{aligned}
\end{equation*}
The last term on the right-hand side is estimated by
\begin{align*}
\mu\iint_{Q_2}v|\z_t|&\,dxdt\le\frac{|Q_2|\mu}{\bar\theta\rho^p}\dashint\!\!\dashint_{Q_2}v\,dxdt\\
&\le\gm\mu\bar\theta\rho^N\bigg[\sup_{-\bar\theta\rho^p<t<0}\dashint_{K_{2\rho}}v(x,t)\,dx\bigg]\ \frac1{\bar\theta}\\
&=\gm\mu^{\frac1{p-1}}\bar\theta^{\frac1{p-1}}\rho^{\frac N{p-1}}\bigg[\sup_{-\bar\theta\rho^p<t<0}\dashint_{K_{2\rho}}v(x,t)\,dx\bigg]\ \mu^{\frac{p-2}{p-1}}\bar\theta^{\frac{p-2}{p-1}}\rho^{\frac {N(p-2)}{p-1}}\frac1{\bar\theta}\\
&\le C_{\eta_1}\mu\theta\rho^N\bigg[\sup_{-\bar\theta\rho^p<t<0}\dashint_{K_{2\rho}} v(x,t)\,dx\bigg]^{p-1}
+\eta_1\mu\bar\theta\rho^N\frac1{\bar\theta^{\frac{p-1}{p-2}}},
\end{align*}
where $\eta_1\in(0,1)$ will be determined later.
The second term is estimated by Lemma \ref{Lm:2:2} choosing $\sig$ such that
\[
p-2+\sig\left(1+\frac pN\right)=p-1.
\]
We have
\begin{align*}
\iint_{Q_2}v^p|D\z|^p\,dxdt&\le\frac{\mu}{\rho^p}\iint_{Q_2} v^{p-1}\,dxdt\\
&\le\gm\frac{\mu |Q_2|}{\rho^p}\dashint\!\!\dashint_{Q_2} v^{p-1}\,dxdt\\
&\le\gm\mu\bar\theta\rho^N\bigg[C_{\eta_2}\left(\sup_{-\bar\theta\rho^p<t<0}\dashint_{K_{2\rho}} v(x,t)\,dx\right)^{p-1}+\frac{\eta_2}\gm\frac1{\bar\theta^{\frac{p-1}{p-2}}}\bigg]\\
&=\gm C_{\eta_2}\mu\bar\theta\rho^N\left(\sup_{-\bar\theta\rho^p<t<0}\dashint_{K_{2\rho}} v(x,t)\,dx\right)^{p-1}+
\eta_2\mu\bar\theta\rho^N\frac1{\bar\theta^{\frac{p-1}{p-2}}},
\end{align*}
where once more $\eta_2\in(0,1)$ will be determined later.
Next, the first term is estimated by H\"older's inequality
\begin{equation*}
\begin{aligned}
\iint_{Q_2}&|Dv|^{p-1}\,dxdt\\
&\le \bigg(\iint_{Q_2}|Dv|^p v^{-1-\eps}\,dxdt\bigg)^{1-\frac1p}
\bigg(\iint_{Q_2}v^{(1+\eps)(p-1)}\,dxdt\bigg)^{\frac1p},
\end{aligned}
\end{equation*}
and the term with the gradient is estimated by Lemma \ref{Lm:2:1}, namely
\begin{equation*}
\begin{aligned}
&\iint_{Q_2}|Dv|^p v^{-1-\eps}\,dxdt\\
&\le \gm(\eps)\bigg[\iint_{Q_3} v^{p-\eps-1}|D\vp|^p\,dxdt
+\iint_{Q_3} v^{1-\eps}|\vp_t|\,dxdt\bigg],
\end{aligned}
\end{equation*}
where $\eps$ is a positive number such that
\[0<\eps<p-2.\]
Combining the above two inequalities yields
\begin{align*}
&\frac{\mu}{\rho}\iint_{Q_2}|Dv|^{p-1}\,dxdt\\
&\le\frac{\gm\mu}{\rho}\bigg(\iint_{Q_3} v^{p-\eps-1}|D\vp|^p\,dxdt\bigg)^{1-\frac1p}
\bigg(\iint_{Q_2}v^{(1+\eps)(p-1)}\,dxdt\bigg)^{\frac1p}\\
&+\frac{\gm\mu}{\rho}\bigg(\iint_{Q_3} v^{1-\eps}|\vp_t|\,dxdt\bigg)^{1-\frac1p}
\bigg(\iint_{Q_2}v^{(1+\eps)(p-1)}\,dxdt\bigg)^{\frac1p}.
\end{align*}
Let us focus on the first term on the right-hand side. Choosing $\eps$ smaller if necessary, and $\sig$ such that 
first $p-\eps-1=p-2+\sig\left(1+\frac pN\right)$ and then $(1+\eps)(p-1)=p-2+\sig\left(1+\frac pN\right)$,
by Lemma \ref{Lm:2:2} and repeated applications of Young's inequality we have
\begin{align*}
&\frac{\gm\mu}{\rho}\bigg[\frac{|Q_3|}{\rho^p}\dashint\!\!\dashint_{Q_3} v^{p-\eps-1}\,dxdt\bigg]^{1-\frac1p}
\bigg[|Q_2|\dashint\!\!\dashint_{Q_2}v^{(1+\eps)(p-1)}\,dxdt\bigg]^{\frac1p}\\
&\le\frac{\gm\mu}\rho\left[\frac{|Q_3|}{\rho^{p}}\right]^{\frac{p-1}p}
\bigg[C_{\eta_3}\left(\sup_{-\bar\theta\rho^p<t<0}\dashint_{K_{2\rho}} v(x,t)\,dx\right)^{p-\eps-1}+\eta_3^{\frac p{p-1}}\left(\frac1{\bar\theta}\right)^{\frac{p-\eps-1}{p-2}}\bigg]^{\frac{p-1}p}\\
&\quad\times|Q_2|^{\frac1p}\bigg[C_{\eta_4}\left(\sup_{-\bar\theta\rho^p<t<0}\dashint_{K_{2\rho}} v(x,t)\,dx\right)^{(1+\eps)(p-1)}+\eta_4^{p}\left(\frac1{\bar\theta}\right)^{\frac{(1+\eps)(p-1)}{p-2}}\bigg]^{\frac1p}\\
&\le\gm\mu\bar\theta\rho^N\bigg[C_{\eta_3}\left(\sup_{-\bar\theta\rho^p<t<0}\dashint_{K_{2\rho}} v(x,t)\,dx\right)^{p-\eps-1}+\eta_3^{\frac p{p-1}}\left(\frac1{\bar\theta}\right)^{\frac{p-\eps-1}{p-2}}\bigg]^{\frac{p-1}p}\\
&\quad\times\bigg[C_{\eta_4}\left(\sup_{-\bar\theta\rho^p<t<0}\dashint_{K_{2\rho}} v(x,t)\,dx\right)^{(1+\eps)(p-1)}+\eta_4^p\left(\frac1{\bar\theta}\right)^{\frac{(1+\eps)(p-1)}{p-2}}\bigg]^{\frac1p}\\
&\le\gm\mu\bar\theta\rho^N\bigg[C_{\eta_3}^{\frac{p-1}p}\left(\sup_{-\bar\theta\rho^p<t<0}\dashint_{K_{2\rho}} v(x,t)\,dx\right)^{\frac{(p-\eps-1)(p-1)}p}+\eta_3\left(\frac1{\bar\theta}\right)^{\frac{(p-\eps-1)(p-1)}{p(p-2)}}\bigg]\\
&\quad\times\bigg[C_{\eta_4}^{\frac1p}\left(\sup_{-\bar\theta\rho^p<t<0}\dashint_{K_{2\rho}} v(x,t)\,dx\right)^{\frac{(1+\eps)(p-1)}p}+\eta_4\left(\frac1{\bar\theta}\right)^{\frac{(1+\eps)(p-1)}{p(p-2)}}\bigg]\\
&\le\gm\mu\bar\theta\rho^N\bigg[C_{\eta_5}\left(\sup_{-\bar\theta\rho^p<t<0}\dashint_{K_{2\rho}} v(x,t)\,dx\right)^{p-1}+\frac{\eta_5}{\gm}\left(\frac1{\bar\theta}\right)^{\frac{p-1}{p-2}}\bigg]\\
&=\gm C_{\eta_5}\mu\bar\theta\rho^N\left(\sup_{-\bar\theta\rho^p<t<0}\dashint_{K_{2\rho}} v(x,t)\,dx\right)^{p-1}+{\eta_5}\mu\bar\theta\rho^N \frac1{\bar\theta^{\frac{p-1}{p-2}}},
\end{align*} 
where once more $\eta_5\in(0,1)$ will be chosen later. The second term on the right-hand side is estimated by
\begin{align*}
&\frac{\gm\mu}{\rho}\bigg(\iint_{Q_3} v^{1-\eps}|\vp_t|\,dxdt\bigg)^{1-\frac1p}
\bigg(|Q_2|\dashint\!\!\dashint_{Q_2}v^{(1+\eps)(p-1)}\,dxdt\bigg)^{\frac1p}\\
&\le\frac{\gm\mu}{\rho}\bigg(\bigg[|Q_3|\dashint\!\!\dashint_{Q_3} v\,dxdt\bigg]^{1-\eps}\frac{|Q_3|^{\eps}}{\bar\theta\rho^p}\bigg)^{1-\frac1p}\\
&\times|Q_3|^{\frac1p}\bigg[C_{\eta_6}\left(\sup_{-\bar\theta\rho^p<t<0}\dashint_{K_{2\rho}} v(x,t)\,dx\right)^{(1+\eps)(p-1)}+\eta_6^{p}\left(\frac1{\bar\theta}\right)^{\frac{(1+\eps)(p-1)}{p-2}}\bigg]^{\frac{1}{p}}\\
&\le\gm\mu\bar\theta\rho^N\bigg[\left(\sup_{-\bar\theta\rho^p<t<0}\dashint_{K_{2\rho}} v(x,t)\,dx\right)^{1-\eps}\frac1{\bar\theta}\bigg]^{\frac{p-1}{p}}\\
&\times\bigg[C_{\eta_6}\left(\sup_{-\bar\theta\rho^p<t<0}\dashint_{K_{2\rho}} v(x,t)\,dx\right)^{(1+\eps)(p-1)}+\eta_6^{p}\left(\frac1{\bar\theta}\right)^{\frac{(1+\eps)(p-1)}{p-2}}\bigg]^{\frac{1}{p}}\\
&\le\gm\mu\bar\theta\rho^N
\bigg[C_{\eta_7}\left(\sup_{-\bar\theta\rho^p<t<0}\dashint_{K_{2\rho}} v(x,t)\,dx\right)^{p-\eps-1}+\eta_7^{\frac p{p-1}}\left(\frac1{\bar\theta}\right)^{\frac{p-\eps-1}{p-2}}\bigg]^{\frac{p-1}{p}}\\
&\times\bigg[C_{\eta_6}\left(\sup_{-\bar\theta\rho^p<t<0}\dashint_{K_{2\rho}} v(x,t)\,dx\right)^{(1+\eps)(p-1)}+\eta_6^{p}\left(\frac1{\bar\theta}\right)^{\frac{(1+\eps)(p-1)}{p-2}}\bigg]^{\frac{1}{p}}\\
&\le\gm\mu\bar\theta\rho^N\bigg[C_{\eta_8}\left(\sup_{-\bar\theta\rho^p<t<0}\dashint_{K_{2\rho}} v(x,t)\,dx\right)^{p-1}+\frac{\eta_8}{\gm}\left(\frac1{\bar\theta}\right)^{\frac{p-1}{p-2}}\bigg]\\
&=\gm C_{\eta_8}\mu\bar\theta\rho^N\left(\sup_{-\bar\theta\rho^p<t<0}\dashint_{K_{2\rho}} v(x,t)\,dx\right)^{p-1}+{\eta_8}\mu\bar\theta\rho^N \frac1{\bar\theta^{\frac{p-1}{p-2}}},
\end{align*}
where as before $\eta_8\in(0,1)$ is still to be chosen.
Therefore, combining all the above estimates we arrive at
\[
\begin{aligned}
\iint_{Q_2}|D(v\z)|^p\,dxdt&\le\bar C\mu\bar\theta\rho^N\bigg[\sup_{-\bar\theta\rho^p<t<0}\dashint_{K_{2\rho}} v(x,t)\,dx\bigg]^{p-1}\\
&+(\eta_1+\eta_2+\eta_5+\eta_8)\mu\bar\theta\rho^N\frac1{\bar\theta^{\frac{p-1}{p-2}}},
\end{aligned}
\]
where $\bar C$ takes into account all the $\gm C_{\eta_i}$-terms.
On the other hand, the left-hand side is bounded from below as
\begin{align*}
\iint_{Q_2}|D(v\z)|^p\,dxdt&\ge\mu^{p}\gm_p(Q_1\setminus E_T,Q_2)\\
&=\mu^{p}\int^0_{-\bar\theta\rho^p}{\rm cap}_p(K_{\rho}\setminus E, K_{\frac32\rho})
\chi_{(-\frac34\bar\theta\rho^p,-\frac14\bar\theta\rho^p)}(t)\,dt\\
&=\frac12\mu^p\bar\theta\rho^p{\rm cap}_p(K_{\rho}\setminus E, K_{\frac32\rho}).
\end{align*}
Thus, recalling the definition of $\dl(\rho)$, we obtain
\[
\begin{aligned}
\mu^{p-1}\frac{{\rm cap}_p(K_{\rho}\setminus E, K_{\frac32\rho})}{{\rm cap}_p(K_{\rho}, K_{\frac32\rho})}\le&\gm
\bigg[\sup_{-\bar\theta\rho^p<t<0}\dashint_{K_{2\rho}} v(x,t)\,dx\bigg]^{p-1}\\
&+\gm(\eta_1+\eta_2+\eta_5+\eta_8)\mu^{p-1}\frac{{\rm cap}_p(K_{\rho}\setminus E, K_{\frac32\rho})}{{\rm cap}_p(K_{\rho}, K_{\frac32\rho})}.
\end{aligned}
\]
Choosing $\eta_1$, $\eta_2$, $\eta_5$, $\eta_8$ such that $\gm(\eta_1+\eta_2+\eta_5+\eta_8)\le\frac12$, 
the above estimate yields
\[
\mu [\dl(\rho)]^{\frac1{p-1}}\le\gm\sup_{-\bar\theta\rho^p<t<0}\dashint_{K_{2\rho}} v(x,t)\,dx.
\]
\end{proof}
We conclude this section, with a second lemma, which will be crucial in the proof of our main result.
\begin{lemma}\label{Lm:3:3}
Let $\pto$, $Q$, $u_k$, $\mu$, $v$ as in \eqref{Eq:cyl1}--\eqref{Eq:super}, take $\bar\theta$ as in \eqref{Eq:theta-bar}, and assume that $s\le t_o-3\gm_*\bar\theta\rho^p<t_o\le t$
for some $\gm_*>1$ to be determined only in terms of the data $\datap$. 
Then there exists a constant $\gm_2>1$, that depends only on the data $\datap$, such that
\begin{equation}\label{Eq:low-bd3}
\mu [\dl(\rho)]^{\frac{1}{p-1}}\le
\gm_2\inf_{K_{2\rho}(x_o)} v(\cdot,t),
\end{equation}
for all $\dsty t\in[t_o-\gm_*\bar\theta\rho^p,t_o]$.
\end{lemma}
\begin{proof}
We may assume that $(x_o,t_o)=(0,0)$. By our notion \eqref{Eq:1:4} of solutions, it
is not hard to verify that 
\[
[-\gamma_*\bar{\theta}\rho^p,0]\ni t\to\dashint_{K_{2\rho}}v(x,t)\,dx\,\text{ is a continuous function.}
\]
Let $t_1\in[-\bar\theta\rho^p,0]$ be the point where the supremum in \eqref{Eq:low-bd2} is achieved, namely
\begin{equation}\label{Eq:step:1}
\mu [\dl(\rho)]^{\frac1{p-1}}\le\gm_1\dashint_{K_{2\rho}} v(x,t_1)\,dx.
\end{equation}
On the other hand, by the weak Harnack inequality \eqref{WHI}, we have
\begin{equation}\label{Eq:step:2}
\bint_{K_{2\rho}}v(x,t_1)dx\le\bar\gm\inf_{K_{8\rho}}v(\cdot,t)
\end{equation}
for any $\dsty t\in[t_1+\frac12\tilde\theta\rho^p,t_1+\tilde\theta\rho^p]$,
where
\[
\tilde\theta=\left[\bint_{K_{2\rho}}v(x,t_1)dx\right]^{2-p}\rho^p.
\]
Combining \eqref{Eq:step:1} and \eqref{Eq:step:2} yields
\begin{equation}\label{Eq:step:3}
\frac1{\bar\gm \gm_1}\mu [\dl(\rho)]^{\frac1{p-1}}\le\inf_{K_{8\rho}}v(\cdot,t)
\end{equation}
for any $\dsty t\in[t_1+\frac{\gm_*}2\bar\theta\rho^p,t_1+\gm_*\bar\theta\rho^p]$
with $\gm_*=\gm_1^{p-2}$. 
At this stage, the time interval where the infimum is taken is somewhat undefined, 
since a precise value of $t_1$ is not known. The next argument is meant to 
provide a precise localization in time of a lower bound for $v$. 

By its definition, $t_1+\gm_*\bar\theta\rho^p\ge0$. On the other hand,
\[
t_1+\gm_*\bar\theta\rho^p= t_1+\gm_*([\dl(\rho)]^{\frac1{p-1}}\mu)^{2-p}\rho^p\le \gm_*([\dl(\rho)]^{\frac1{p-1}}\mu)^{2-p}\rho^p.
\]
Therefore, if we apply Lemma~\ref{LBL1} with $\bar t=t_1+\gm_*\bar\theta\rho^p$, and take
\[
t\in[\gm_*([\dl(\rho)]^{\frac1{p-1}}\mu)^{2-p}\rho^p,2\gm_*([\dl(\rho)]^{\frac1{p-1}}\mu)^{2-p}\rho^p],
\]
we have $t-\bar t\le 2\gm_*([\dl(\rho)]^{\frac1{p-1}}\mu)^{2-p}\rho^p$, and substituting in \eqref{LB}, we conclude, where $\gm_2$ depends on $\nu$, $\bar\gm$, $\gm_1$, and $p$.
\end{proof}
\section{Proof of Theorem~\ref{Thm:1:1}}\label{S:final}
Let $\pto\in S_T$, and for $R_o>0$ set
\[
Q_{R_o}=K_{2R_o}(x_o)\times(t_o-3\gm_*[\dl(R_o)]^{\frac{2-p}{p-1}}R_o^{p-\eps},t_o],
\]
where $0<\eps<1$ and $\dl(R_o)$ has been defined in \eqref{Eq:delta}.
As discussed in \S~\ref{S:intro}, we may take $R_o$ so small that 
$$(t_o-3\gm_*[\dl(R_o)]^{\frac{2-p}{p-1}}R_o^{p-\eps},t_o]\subset(0,T].$$
Next, if we choose the level
\[
k=\sup_{Q_{R_o}\cap S_T}g,
\]
then Lemma~\ref{Lm:2:0} can be applied. From now on, we deal with such a level, and with the corresponding truncated function $u_k\df=(u-k)_+$. Moreover, we assume that $u_k$ has been extended to zero in $Q_{R_o}\backslash E_T$.
\subsection{The First Step}
Consider $u_{k}$, and choose $\mu_o>0$ such that 
\begin{equation}\label{Eq:1step}
\mu_o=\sup_{Q_{R_o}}u_{k}.
\end{equation}
Without loss of generality, we may assume that 
\begin{equation}\label{Eq:in-req}
\mu_o^{2-p}R_o^p\le R_o^{p-\eps}.
\end{equation}
Indeed, if \eqref{Eq:in-req} is not satisfied, then $\mu_o$ has a power-like decay with respect to $R_o$, and there is nothing to prove. If we let
\[
v\df=\mu_o-u_{k},
\]
and
\[
\bar\theta_o\df=\left(\mu_o[\dl(R_o)]^{\frac1{p-1}}\right)^{2-p},
\]
by \eqref{Eq:in-req}, the assumptions of Lemma~\ref{Lm:3:3} are satisfied, and we conclude that
\[
\mu_o[\dl(R_o)]^{\frac1{p-1}}\le\gm_2\inf_{K_{2R_o}(x_o)}v(\cdot,t)
\]
for all $t\in\bigg[t_o-\gm_*\left(\mu_o[\dl(R_o)]^{\frac1{p-1}}\right)^{2-p}R_o^p,t_o\bigg]$, that is
\begin{equation}\label{Eq:first-bound}
\sup_{Q_1} u_{k}\le\mu_o\left(1-\frac1{\gm_2}[\dl(R_o)]^{\frac1{p-1}}\right),
\end{equation}
where
\begin{equation}\label{Eq:first-cyl}
\begin{aligned}
Q_1&=K_{2R_o}(x_o)\times\bigg[t_o-\gm_*\left(\mu_o[\dl(R_o)]^{\frac1{p-1}}\right)^{2-p}R_o^{p},t_o\bigg]\\
       &=K_{2R_o}(x_o)\times[t_o-\gm_*\bar\theta_oR_o^{p},t_o].
\end{aligned}
\end{equation}
\subsection{The Induction}
We now proceed by induction. In order to do that, we first need the following result
which is based on the fact that we assume {\it a priori} the Wiener integral \eqref{Eq:fat}
is divergent. The idea of selecting a specific subsequence is taken from \cite{skrypnik-2000}.
For ease of notation, we set $A(s)=[\dl(s)]^{\frac1{p-1}}$.
\begin{lemma}\label{Lm:4:1}
Assume that 
\[
\int_0^1A(s)\frac{ds}s=\infty.
\]
Then there exist $\bar c\in(0,1)$ depending only the data,
 and a subsequence $\{\rho_{i_j}\}$ of
the sequence $\{\rho_i=\bar c^i R_o\}$, such that
\begin{equation}\label{Eq:bar-c}
3\left[\mu_{i_{j+1}}A(\rho_{i_{j+1}})\right]^{2-p}\rho_{i_{j+1}}^{p}
\le\left[\mu_{i_{j}}A(\rho_{i_{j}})\right]^{2-p}\rho_{i_{j}}^{p},
\end{equation}
where
\[
\mu_{i_{j+1}}=\mu_{i_j}\left[1-\frac1{\gm_2}A(\rho_{i_j})\right]. 
\]
Moreover,
\begin{equation}\label{Eq:sub-bd}
\sum_{i=0}^{i_{k+1}-1}A(\rho_i)\le 2\sum_{j=0}^{k}A(\rho_{i_j})\quad\text{ for any }k=0,1,2,\cdots.
\end{equation}
\end{lemma}
\begin{proof}
First, we observe that the divergence of the Wiener integral implies the divergence of the series
\[\sum_{i=0}^{\infty}A(\rho_i),\]
which does not require any quantitative information about $\bar c$.
 Next, for any non-negative integer $i$, there exists $j\in\nn$ such that
\[
\frac{A(\rho_{i+j})}{A(\rho_i)}>\left(\frac12\right)^j;
\]
otherwise, it would lead to the convergence of the series $\displaystyle \sum_{i=0}^{\infty}A(\rho_i)$.

\noindent Let $i_o=0$ and choose $i_1>i_o$ to be the smallest positive integer satisfying
\[
\frac{A(\rho_{i_1})}{A(\rho_{i_o})}>\left(\frac12\right)^{i_1-i_o};
\]
by induction, we obtain a subsequence $\{i_j\}$ such that
$i_{j+1}>i_j$ is the smallest positive integer satisfying
\begin{equation}\label{Eq:Aij}
\frac{A(\rho_{i_{j+1}})}{A(\rho_{i_j})}>\left(\frac12\right)^{i_{j+1}-i_j}.
\end{equation}
Next, we observe that 
\[
3\left[\mu_{i_{j+1}}A(\rho_{i_{j+1}})\right]^{2-p}\rho_{i_{j+1}}^{p}
\le3\left[\left(1-\frac1{\gm_2}\right)\mu_{i_j}A(\rho_{i_{j+1}})\right]^{2-p}\rho_{i_{j+1}}^{p}.
\]
Hence, in order to show \eqref{Eq:bar-c}, we need only to show
\[
3\left[\left(1-\frac1{\gm_2}\right)A(\rho_{i_{j+1}})\mu_{i_j}\right]^{2-p}\rho_{i_{j+1}}^{p}
\le\left[\mu_{i_{j}}A(\rho_{i_{j}})\right]^{2-p}\rho_{i_{j}}^{p}.
\]
This is equivalent to 
\[
\frac{A(\rho_{i_{j+1}})}{A(\rho_{i_j})}\ge3^{\frac1{p-2}}\left(1-\frac1{\gm_2}\right)^{-1}\bar c^{\frac{p}{p-2}(i_{j+1}-i_j)}.
\]
Comparing this with \eqref{Eq:Aij}, one easily obtains \eqref{Eq:bar-c} by choosing $\bar c=2^{-\lm}$
with some large $\lm$ satisfying
\[
2^{\frac{\lm p}{p-2}-1}\ge 3^{\frac1{p-2}}\left(1-\frac1{\gm_2}\right)^{-1}.
\]
Finally, according to the way of choosing $i_{j+1}$, we must have
\begin{equation*}
\frac{A(\rho_i)}{A(\rho_{i_j})}\le\left(\frac12\right)^{i-i_j}\quad\text{ for any }i_j\le i\le i_{j+1}-1.
\end{equation*}
This implies
\[
\sum_{i=i_j}^{i_{j+1}-1}A(\rho_i)\le A(\rho_{i_j})\sum_{i=i_j}^{\infty}\left(\frac12\right)^{i-i_j}\le 2A(\rho_{i_j}).
\]
Summing the above inequality over $j$ from $0$ to $k$ yields 
\[
\sum_{i=0}^{i_{k+1}-1}A(\rho_i)
=\sum_{j=0}^{k}\sum_{i=i_j}^{i_{j+1}-1}A(\rho_i)\le 2\sum_{j=0}^{k}A(\rho_{i_j})\quad\text{ for any }k=0,1,2,\cdots.
\]
This concludes the proof.
\end{proof}
Now, assume that up to step $l$ we have shown
\[
\sup_{Q_{i_j}}u_k\le \mu_{i_j}\qquad j=1,\dots,l,
\]
where
\[
Q_{i_j}=K_{2\rho_{i_{j-1}}}(x_o)\times(t_o-\gm_*\bar\theta_{i_{j-1}}\rho_{i_{j-1}}^p,t_o]
\]
and
\[
\bar\theta_{i_{j-1}}=\left[\mu_{i_{j-1}}A(\rho_{i_{j-1}})\right]^{2-p},\ \ \mu_{i_j}=\mu_{i_{j-1}}\left[1-\frac1{\gm_2}A(\rho_{i_{j-1}})\right].
\]
Then by \eqref{Eq:bar-c} and Lemma~\ref{Lm:3:3} we have 
\[
\sup_{Q_{i_{l+1}}}u_{k}\le \mu_{i_{l+1}},
\]
where
\[
Q_{i_{l+1}}=K_{2\rho_{i_l}}(x_o)\times(t_o-\gm_*\bar\theta_{i_l}\rho_{i_l}^p,t_o]
\]
and
\[
\bar\theta_{i_l}=\left[\mu_{i_l}A(\rho_{i_l})\right]^{2-p},\ \ \mu_{i_{l+1}}=\mu_{i_l}\left[1-\frac1{\gm_2}A(\rho_{i_l})\right].
\]
Employing \eqref{Eq:sub-bd}, 
we can now conclude as in \cite[Section~6.4]{GLL}: 
there exists a constant $\gm_3>1$ that depends only on the data $\datap$, such that
\begin{align*}
\sup_{Q_{i_{l+1}}}u_{k}&\le \mu_{i_l} \left[1-\frac{1}{\gm_2}A(\rho_{i_l})\right]\\
&\le \mu_o\exp\left\{-\frac{1}{\gm_2}\sum_{j=0}^{l}A(\rho_{i_j})\right\}\\
&\le \mu_o\exp\left\{-\frac1{2\gm_2}\sum_{i=0}^{i_{l+1}-1}A(\rho_i)\right\}\\
&\le\mu_o\exp\left\{-\frac1{\gm_3}\int_{\rho_{i_{l+1}}}^{R_o}A(s)\frac{ds}s\right\};\\
\end{align*}
taking into consideration the reverse case of \eqref{Eq:in-req} actually yields that
\begin{equation}\label{Eq:final1}
\sup_{Q_{i_{l+1}}}\,(u-k)_+\le\mu_o\exp\left\{-\frac1{\gm_3}\int_{\rho_{i_{l+1}}}^{R_o}A(s)\frac{ds}s\right\}
+\gm_3 R_o^{\frac{\eps}{p-2}}.
\end{equation}
Now fix $\rho\in(0,R_o)$; there is an integer $l\ge0$ such that
\[
\rho_{i_{l+1}}\le\rho<\rho_{i_l}.
\]
As a result, it is easy to check that
\[
Q_{\rho}(\mu_o)=K_{2\rho}(x_o)\times [t_o-\mu_o^{2-p}\rho^p,t_o]\subset Q_{i_{l+1}}.
\]
Hence, we may conclude from \eqref{Eq:final1} that
\begin{equation}\label{Eq:final2}
\sup_{Q_{\rho}(\mu_o)}\,(u-k)_+\le\mu_o\exp\left\{-\frac1{\gm_3}\int_{\rho}^{R_o}A(s)\frac{ds}s\right\}+\gm_3 R_o^{\frac{\eps}{p-2}}.
\end{equation}
Similarly, if we set 
$$h=\inf_{Q_{R_o}}g,$$ 
and work with $u_{h}=(h-u)_+$,
an analogous argument as above gives that
\begin{equation}\label{Eq:final3}
\sup_{Q_{\rho}(\tilde{\mu}_o)}\,(h-u)_+\le\tilde{\mu}_o\exp\left\{-\frac1{\gm_3}\int_{\rho}^{R_o}A(s)\frac{ds}s\right\}+\gm_3 R_o^{\frac{\eps}{p-2}},
\end{equation}
where
\[\tilde{\mu}_o= \sup_{Q_{R_o}}u_h.\]
Note that 
\[\max\{\mu_o,\,\tilde{\mu}_o\}\le\mu_o+\tilde{\mu}_o\le\om_o-\osc_{Q_{R_o}}g\le\om_o.\]
Combining \eqref{Eq:final2} and \eqref{Eq:final3} yields 
\[
\osc_{Q_{\rho}(\om_o)\cap E_T}\,u\le\om_o\exp\left\{-\frac1{\gm_3}\int_{\rho}^{R_o}A(s)\frac{ds}s\right\}+\osc_{Q_{R_o}\cap S_T}g+2\gm_3 R_o^{\frac{\eps}{p-2}}.
\]
\hfill\os


\begin{thebibliography}{99}
\bibitem{Biroli-Mosco} M. Biroli and U. Mosco, Wiener estimates for parabolic obstacle problems, {\it Nonlinear Anal., 11(9), (1987), 1005--1027.}
\bibitem{BBGP} A. Bj\"orn, J. Bj\"orn, U. Gianazza and M. Parviainen, Boundary regularity for degenerate and singular parabolic equations, {\it Calc. Var. Partial Differential Equations, 52(3), (2015), 797--827.}
\bibitem{dibe-sv} E. DiBenedetto, {\it Degenerate Parabolic 
Equations}, Universitext, Springer-Verlag, New York, 1993.  
\bibitem{DBGV-mono} E. DiBenedetto, U. Gianazza and V. Vespri, 
{\it Harnack's Inequality for Degenerate and Singular Parabolic 
Equations}, Springer Monographs in Mathematics, Springer-Verlag, 
New York, 2012.
\bibitem{Evans-Gariepy} L.C. Evans and R.F. Gariepy, {\it Measure Theory and Fine Properties of Functions}, Studies in Advanced Mathematics CRC Press, Boca Raton, 1992.
\bibitem{frehse} J. Frehse, {Capacity methods in the theory of partial differential equations},
{\it Jahresber. Deutsch. Math.-Verein., 84(1), (1982), 1--44.}
\bibitem{GLL} U. Gianazza, N. Liao and T. Lukkari, A boundary estimate for singular parabolic diffusion equations, 
{\it Nonlinear Differ. Equ. Appl. (2018) 25:33.}
\bibitem{GSV} U. Gianazza, M. Surnachev and V. Vespri, On a new 
proof of H\"older continuity of solutions of $p$-Laplace 
type parabolic equations, {\it Adv. Calc. Var., {3}(3), (2010), 263--278.}
\bibitem{HKM} J. Heinonen, T. Kilpel\"ainen and O. Martio, {\it Nonlinear Potential Theory of Degenerate Elliptic Equations}, 2nd ed., Dover, Mineola, NY, 2006.
\bibitem{KiLi96} T. {Kilpel\"ainen} and P. {Lindqvist}, {On the Dirichlet boundary value problem for a degenerate parabolic equation},  {\it SIAM J. Math. Anal., 27,  (1996), 661--683.}
\bibitem{kuusi2008} T. Kuusi, Harnack estimates for weak 
supersolutions to nonlinear 
degenerate parabolic equations, {\it Ann. Scuola Norm. Sup. Pisa 
Cl. Sci. (5),  {7}(4), (2008), 673--716.}
\bibitem{lewis1988} J. L. Lewis,  Uniformly fat sets, 
{\it Trans. Amer. Math. Soc., 308, (1988), 177--196.}
\bibitem{Ma-Zie} J. Mal\'y and W.P. Ziemer, {\it Fine Regularity of Solutions of Elliptic Partial Differential
Equations}, {Mathematical Surveys and Monographs}, {51}, {American Mathematical Society, Providence, RI},
1997.
\bibitem{skrypnik-2000} I.I. Skrypnik, Regularity of a boundary point for degenerate parabolic equations with measurable coefficients, {\it Ukra\"\i n. Mat. Zh., 52(11), (2000), 1550--1565 (in Russian). English transl.: Ukrainian Math. J., 52(11), (2000), 1768--1786.}
\bibitem{skrypnik-2004} I.I. Skrypnik, {Regularity of a boundary point for singular parabolic equations with measurable coefficients}, {\it Ukra\"\i n. Mat. Zh., 56(4), (2004), 506--516 (in Russian). English transl.:  Ukrainian Math. J., 56(4), (2004), 614--627.}  
\bibitem{ziemer} W.P. Ziemer, Behavior at the boundary of solutions of quasilinear parabolic equations, {\it J. Differential Equations, 35(3), (1980), 291--305.}
\bibitem{ziemer1982} W.P. Ziemer, Interior and boundary continuity of weak solutions of degenerate parabolic equations, {\it Trans. Amer. Math. Soc., 271(2), (1982), 733--748.}
\end{thebibliography}
\end{document}